\documentclass{amsart} 

\usepackage{amsthm,amsfonts,amsmath,amssymb,enumitem}
\usepackage{mathtools,mathrsfs}
\usepackage{tikz-cd,graphicx}
\usepackage{thmtools,thm-restate}
\usepackage[noadjust]{cite}
\usepackage[margin=1in]{geometry}
\usepackage[symbol]{footmisc}

\usepackage{soul}
\sethlcolor{lightgray}

\newtheorem{theorem}{Theorem}[section]
\newtheorem{prop}[theorem]{Proposition}
\newtheorem{lemma}[theorem]{Lemma}

\theoremstyle{definition}

\newtheorem*{acks}{Acknowledgments}
\newtheorem{definition}[theorem]{Definition}

\theoremstyle{remark}

\newtheorem{remark}[theorem]{Remark}

\newtheorem{ex}[theorem]{Example}

\newcommand{\E}{{\mathbb{E}}}
\newcommand{\F}{{\mathbb{F}}}
\newcommand{\bbH}{{\mathbb{H}}}
\newcommand{\I}{{\mathbb{I}}}
\newcommand{\N}{{\mathbb{N}}}

\newcommand{\Z}{{\mathbb{Z}}}

\newcommand{\cC}{{\mathcal{C}}}
\newcommand{\cF}{{\mathcal{F}}}

\newcommand{\lk}{{\mathrm{lk}}}

\newcommand{\sm}{{\mathrm{sm}}}
\newcommand{\ttop}{{\mathrm{top}}}

\title{New examples of topologically slice links}
\author{Alex Manchester}
\thanks{The author was partially supported by the RTG award NSF DMS-1745670.}

\begin{document}

\begin{abstract}
In 2007, Cochran-Friedl-Teichner gave sufficient conditions for when a link obtained by multi-infection is topologically slice involving a Milnor's invariant condition on the infecting string link. In this paper, we give a different Milnor's invariant condition which can handle some cases which the original theorem cannot. Along the way, we also give sufficient conditions for a multi-infection to be $(n-0.5)$-solvable, where we require that the infecting string link have vanishing pairwise linking numbers, which can be seen as handling an additional ``$(-0.5)$-solvable'' case of a well-known result about the relationship between satellite operations and the solvable filtration of Cochran-Orr-Teichner.
\end{abstract}

\maketitle

\section{Introduction}
\label{Sec:Intro}

The known constructions of slice links are still quite limited in both the smooth and the topological settings. Smoothly, the famous Slice-Ribbon conjecture asserts that every smoothly slice link arises as the boundary of a set of ribbon disks, and thus can be constructed in a standard way (see e.g. \cite{rolfsenKnotsLinks2003}). However, there are some known smoothly slice links which are not known to arise in this way (see \cite{gompfFiberedKnotsPotential2010} and \cite{abeConstructionSliceKnots2015}). Another interesting construction of smoothly slice links can be found in \cite{cochranCounterexamplesKauffmanConjectures2014}, though these are known to arise as the boundary of a set of ribbon disks.

	Freedman showed that any knot with Alexander polynomial 1 (which includes all Whitehead doubles) is topologically slice (see \cite{freedmanNewTechniqueLink1985}), and together with Teichner showed that the Whitehead doubles of certain links are topologically slice (see \cite{freedman4ManifoldTopologyII1995}). It is a corollary of work of Donaldson in \cite{donaldsonApplicationGaugeTheory1983} that the Whitehead double of the right-handed trefoil is not smoothly slice even though it is topologically slice. This was first noticed by Akbulut and Casson, and appeared in the literature in \cite{cochranApplicationsDonaldsonTheorems1988}. In \cite{cochranNewConstructionsSlice2007}, Cochran-Friedl-Teichner generalized many of these constructions via Theorem \ref{Thm:COTThm} below. Before stating the theorem, we will state some necessary definitions.

\begin{definition}
	An {\it $r$-multi-disk} $\E$ is an oriented disk $D^2$ together with $r$ ordered embedded open disks $E_1,...,E_r$ with disjoint closures.
\end{definition}

\begin{definition}
	Given a link $L \subset S^3$, a (smoothly) embedded multi-disk $\E \subset S^3$ is {\it proper} if $L$ intersects the image transversely and only in the the subdisks $E_1,...,E_r$.
\end{definition}

\begin{figure}
	\includegraphics[scale=0.5]{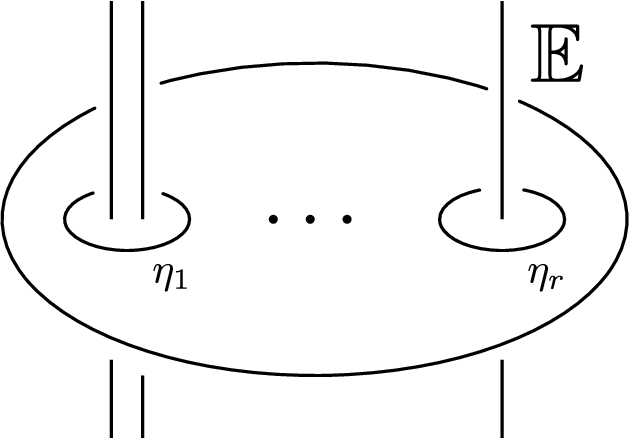}
	\caption{Proper multi-disk}
	\label{Fig:Multidisk}
\end{figure}

\noindent We will denote the boundaries of $E_1,...,E_r$ by $\eta_1,...,\eta_r$. See Fig.\ \ref{Fig:Multidisk}.

Given an unoriented $r$-component string link $J = J_1 \sqcup\cdots\sqcup J_r \subset D^2 \times I$, we can thicken a proper multi-disk $\E$ and replace the resulting handlebody by $(D^2 \times I) \setminus \nu(J)$ via the following homeomorphism on the boundaries: The thickening $\bbH$ of $\E$ can be seen as $(D^2 \times I) \setminus \nu(U)$ where $U$ is trivial string link. The boundary of both $(D^2 \times I) \setminus \nu(J)$ and $(D^2 \times I) \setminus \nu(U)$ consist of the union of a $2r$-punctured sphere which is a portion of the boundary of $D^2 \times I$, with $r$ cylindrical ``tubes'' attached. For the gluing, use the homeomorphism which is the identity on the $2r$-punctured sphere and which takes each tube to the corresponding tube, identifying the trivial framings of $U$ and $J$. The resulting link is called the multi-infection of $L$ by $J$ along $\E$. Intuitively, this amounts to ``tying $J$ into $L$ .'' See Fig.\ \ref{Fig:Multiinfection}.

\begin{figure}
	\includegraphics{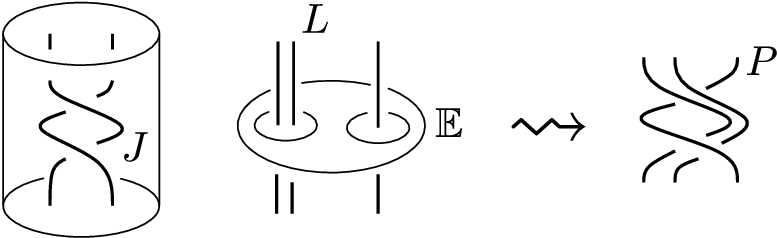}
	\caption{Multi-infection of $L$ by $J$ along $\E$}
	\label{Fig:Multiinfection}
\end{figure}

Many of the known examples of topologically slice links can be produced by multi-infection, as in the following theorem:

\begin{theorem}[Theorem 1.5 of \cite{cochranNewConstructionsSlice2007}]
\label{Thm:COTThm}
Let $L \subset S^3$ be an $m$-component slice link with slice disks $\Delta=\Delta_1\sqcup\cdots\sqcup\Delta_m \subset B^4$. Let $W:=B^4\setminus\nu(\Delta)$. Let $\E$ be a proper $r$-multi-disk in $S^3$ such that $\eta_1,...,\eta_r$ are nullhomotopic in $W$ and thus bound immersed disks $\delta=\delta_1\cup\cdots\cup\delta_r$ with $c$ (self-)intersections. Let $J$ be an $r$-component string link whose closure $\hat{J}$ has Milnor $\bar\mu$-invariants vanishing up to (and including) length \hl{$2c$}.\footnote{We have highlighted the quantities \hl{$2c$} and \hl{$c+r$} throughout to emphasize the difference between Theorems \ref{Thm:COTThm} and \ref{Thm:Slice}. Except for these quantities, the statements are identical.} Then the multi-infection of $L$ by $J$ is topologically slice.
\end{theorem}

\begin{remark}
Throughout, we use the term ``(self-)intersections'' to mean ``intersections and self-intersections.''
\end{remark}

\begin{remark}
We will not define Milnor $\bar\mu$-invariants in this paper; an overview from the perspective relevant here can be found in \cite{freedman4ManifoldTopologyII1995}, and we will ultimately make use of the Milnor $\bar\mu$-invariant condition via an application of Theorem \ref{Thm:MilnorImmersedDisks}, which appears as Lemma 2.7 \cite{freedman4ManifoldTopologyII1995}. The result first appeared as Corollary 9.4 in \cite{cochranDerivativesLinksMilnor1990}.
\end{remark}

Other notable constructions of topologically slice links include \cite{cochranStructureBipolarFiltration2015}, \cite{chaCoveringLinkCalculus2014a}, and \cite{chaCassonTowersSlice2016}.

It has gotten much easier in recent years to show that certain knots and links are not smoothly slice using techniques such as gauge theory (\cite{donaldsonApplicationGaugeTheory1983}, \cite{gompfSmoothConcordanceTopologically1986}, \cite{cochranApplicationsDonaldsonTheorems1988}, \cite{endoLinearIndependenceTopologically1995}, \cite{heddenSatellitesInfiniteRank2021}), Ozsv\'{a}th-Szab\'{o}'s Heegaard-Floer homology and its variants (\cite{manolescuConcordanceInvariantFloer2005}, \cite{satoTopologicallySliceKnots2018}; knot Floer: \cite{nieTopologicallySliceKnots2019}, \cite{daiEquivariantKnotsKnot2022}; involutive: \cite{daiCableFigureeightKnot2022}; for a survey see \cite{homSurveyHeegaardFloer2017}), and Khovanov homology (\cite{rasmussenKhovanovHomologySlice2004},  \cite{piccirilloConwayKnotNot2018a}, \cite{shumakovitchRasmussenInvariantSliceBennequin2018}). However, essentially the only nonclassical tools for showing that knots are not topologically slice are metabelian invariants, in particular Casson-Gordon and twisted Alexander invariants (\cite{kirkConcordanceMutation2001}, \cite{heraldMetabelianRepresentationsTwisted2008}, \cite{millerDistinguishingMutantPretzel2015}) and Cheeger-Gromov's von Neumann $\rho$-invariants (\cite{cochranDerivativesKnotsSecondorder2010})

The main theorem of this paper resembles Theorem \ref{Thm:COTThm} above, but with a different Milnor's invariant condition.

\begin{restatable*}{theorem}{ThmSlice}
\label{Thm:Slice}
Let $L \subset S^3$ be an $m$-component slice link with slice disks $\Delta=\Delta_1\sqcup\cdots\sqcup\Delta_m \subset B^4$. Let $W:=B^4\setminus\nu(\Delta)$. Let $\E$ be a proper $r$-multi-disk in $S^3$ such that $\eta_1,...,\eta_r$ are nullhomotopic in $W$ and thus bound immersed disks $\delta=\delta_1\cup\cdots\cup\delta_r$ with, say, $c$ (self-)intersections. Let $J$ be an $r$-component string link whose closure $\hat{J}$ has Milnor $\bar\mu$-invariants vanishing up to (and including) length \hl{$c+r$}. Then the multi-infection of $L$ by $J$ is topologically slice.
\end{restatable*}

There are examples which satisfy the condition in Theorem \ref{Thm:Slice} but not the conditions in Theorem \ref{Thm:COTThm}. In particular, any time the $\eta_i$ are freely nullhomotopic in $W$, but require at least $r+1$ (self-)intersections in the immersed disks $\delta_i$, a string link $J$ whose closure $\hat{J}$ has Milnor $\bar\mu$-invariants vanishing up to length \hl{$c+r$} but not up to link \hl{$2c$} would suffice since in this case $r<c$. We will describe an example in which this occurs in Section \ref{Sec:Ex}.

In practice we typically have $r \leq c$, but Theorem \ref{Thm:Slice} is not technically speaking stronger than Theorem \ref{Thm:COTThm} since some of the $\delta_i$ might be embedded and disjoint from the other $\delta_i$, in which case we might have $r>c$. However, cases like this do not require the full difficulty of either theorem, and we could apply more elementary techniques to any of the $\eta_i$ which bound such a $\delta_i$ in order to reduce any such case to one where $r \leq c$. How this can be done in the context of Theorem \ref{Thm:Slice} is described in Remark \ref{Rmk:Degenerate} after the proof of Theorem \ref{Thm:Slice} at the end of Section \ref{Sec:Sliceness}.

We can also refine Theorem \ref{Thm:Slice} slightly so that we don't need {\it all} the Milnor's $\bar\mu$-invariants up to length \hl{$c+r$} to vanish, though it is difficult to state the exact condition precisely. Theorem \ref{Thm:COTThm} could also be refined in the same way. We will discuss the modification in Remark \ref{Rmk:MilnorRefinement} after the proof of Theorem \ref{Thm:Slice} and Remark \ref{Rmk:Degenerate} at the end of Section \ref{Sec:Sliceness}.

Along the way, we will also prove the following theorem:

\begin{restatable*}{theorem}{ThmnSolv}
\label{Thm:nSolv}
Let $L \subset S^3$ be an $m$-component smoothly slice link with (smooth) slice disks $\Delta=\Delta_1\sqcup\cdots\sqcup\Delta_m \subset B^4$. Let $W:=B^4\setminus\nu(\Delta)$. Let $\E$ be a proper $r$-multi-disk in $S^3$ such that $\eta_1,...,\eta_r \in \pi_1(W)^{(n)}$. Let $J$ be an $r$-component string link whose closure $\hat{J}$ has pairwise vanishing linking numbers. Then the multi-infection $P=P(J)$ of $L$ by $J$ along $\E$ is $(n-0.5)$-solvable.
\end{restatable*}


We will define and discuss $(h)$-solvability for $h\in\frac{1}{2}\N$ in Section \ref{Sec:Background}, but will briefly highlight here that links coming from Theorem \ref{Thm:nSolv} are candidates for links which are $(n.5)$- but not $(n+1)$-solvable for $n$ an integer. The first such links were found by Otto in \cite{ottoSolvableFiltrationLink2014}, and it is still open whether there are any knots (that is, 1-component links), or even boundary links, with this property.

Theorem \ref{Thm:nSolv} is related to the following theorem which is well-known to the experts. Very similar theorems       are Proposition 3.1 in \cite{cochranStructureClassicalKnot2004} and Theorem 7.1 in \cite{cochranKnotConcordanceHigherorder2009}, and it first appeared in the literature in this form in \cite{burkeInfectionStringLinks2014}:

\begin{theorem}[see Proposition 3.1 of \cite{cochranStructureClassicalKnot2004}, Theorem 7.1 of \cite{cochranKnotConcordanceHigherorder2009}, and Theorem 3.6 of \cite{burkeInfectionStringLinks2014}]
\label{Thm:Burke}
Let $L \subset S^3$ be an $m$-component slice link with slice disks $\Delta=\Delta_1\sqcup\cdots\sqcup\Delta_m \subset B^4$. Let $W:=B^4\setminus\nu(\Delta)$. Let $\E$ be a proper $r$-multi-disk in $S^3$ such that $\eta_1,...,\eta_r\in\pi_1(W)^{(n)}$. Let $J$ be an $r$-component string link whose closure $\hat{J}$ is $(h)$-solvable. Then the multi-infection $P$ of $L$ by $J$ along $\E$ is $(n+h)$-solvable.
\end{theorem}

In particular, it is often convenient to think of links with vanishing pairwise linking number as ``$(-0.5)$-solvable.'' Theorem \ref{Thm:nSolv} can in this way be seen as addressing the ``$h=-0.5$'' case of Theorem \ref{Thm:Burke}.

Recently, in \cite{manchesterActionMazurPattern2022}, the author showed the following related theorem:

\begin{theorem}[Theorem 1.4 of \cite{manchesterActionMazurPattern2022}]
Given homotopically related satellite operators $P_\epsilon$ for $\epsilon\in\{0,1\}$, for any knot $J$ we have that $P_0(J)\#-P_1(J)$ is (1)-solvable. In other words, $P_0 \equiv P_1$ as operators on $\cC/\cF_{(1)}$, where $\cF_{(1)}$ is the subgroup of (1)-solvable knots.
\end{theorem}

This theorem is an approximate relative version of the 1-component case of Theorems \ref{Thm:COTThm} and \ref{Thm:Slice}. There are obstructions to (1.5)-solvability (and in particular sliceness) coming from Casson-Gordon invariants and metabelian $\rho$-invariants, which fail to obstruct this result when $\cF_(1)$ is replaced by some higher term in the Cochran-Orr-Teichner solvable filtration (see Propositions 5.3 and 6.3 in \cite{manchesterActionMazurPattern2022}), so it is possible that the conclusion of this result can be upgraded.

\begin{acks}
I would like to thank Shelly Harvey and Mark Powell for their useful conversations and suggestions.
\end{acks}

\section{Background}
\label{Sec:Background}

	Two ordered oriented links $L_0$ and $L_1$ are {\it (smoothly/topologically) concordant} if there is a disjoint union of annuli (smoothly/locally-flatly) embedded in $S^3 \times I$, each of which is cobounded by corresponding components of $L_0\times\{0\}$ and $-L_1\times\{1\}$. Let $\cC^m_\ttop$ (resp. $\cC^m_\sm$) denote the set of links up to topological (resp. smooth)  concordance.

A(n) (ordered, oriented) link $L$ in $S^3$ is called {\it (smoothly/topologically) slice} if it is the boundary of a (smoothly/locally-flatly) embedded disjoint union of disks in $B^4$, or equivalently if it is (smoothly/topologically) concordant to the unlink. (The topological case requires highly nontrivial topological 4-manifold theory, in particular the existence of normal bundles for topological submanifolds. See Section 9.3 of \cite{freedmanTopology4manifolds1990} and Section 21.4.8 of \cite{behrensDiskEmbeddingTheorem2021}.)
	
We will need the following well-known characterization of topological sliceness:

\begin{prop}[see e.g. Prop 2.1 of \cite{cochranNewConstructionsSlice2007} for a proof]
\label{TopSliceCondition}
An $m$-component link $L \subset S^3$ is topologically slice if and only if its 0-surgery $M_L$ bounds a topological 4-manifold $W$ such that
	\begin{enumerate}
	\item $\pi_1(W)$ is normally generated by the meridians $\mu_i$ of $L$,
	\item $H_1(W)\cong\Z^m$ is generated by $\mu_i$ (in other words the inclusion map $H_1(M_L) \to H_1(W)$ is an isomorphism), and
	\item $H_2(W)=0$.
	\end{enumerate}
\end{prop}

It is often difficult to check whether $M_L$ bounds such a 4-manifold, but sometimes it is easier to show that $M_L$ bounds a 4-manifold satisfying the following weaker condition, due to Cochran-Orr-Teichner in \cite{cochranKnotConcordanceWhitney2003}. Recall that the derived series of a group $G$ is defined inductively as $G^{(0)}:=G$ and $G^{(n+1)}:=[G^{(n)}:G^{(n)}]$, that is, each term is the commutator subgroup of the previous term.

\begin{definition}
\label{Solvability}
An $m$-component link $L$ is called \emph{$(n)$-solvable} for $n\in\N$ if $M_L$ bounds a smooth 4-manifold $W$ such that
\begin{enumerate}
	\item the inclusion map $H_1(M_L) \to H_1(W)$ is an isomorphism,
	\item $H_2(W)$ is freely generated by smoothly embedded surfaces $\{S_i, T_i\}$ with trivial normal bundle such that $S_i$ and $T_i$ intersect once transversely, and otherwise the surfaces are disjoint, and
	\item the inclusion maps $\pi_1(S_i),\pi(T_i)\to\pi(W)$ land in $\pi_1(W)^{(n)}$.
\end{enumerate}
If, additionally, the inclusion maps $\pi_1(S_i)\to\pi_1(W)$ land in $\pi_1(W)^{(n+1)}$, then $L$ is said to be \it{$(n.5)$-solvable}. If $L$ is $(h)$-solvable for some $h\in\frac{1}{2}\N$, the manifold $W$ is called an \emph{$(h)$-solution} for $L$.
\end{definition}

\begin{remark}
Throughout, we will use $n$ to denote an element of $\N$, and $h$ to denote an element of $\frac{1}{2}\N$.
\end{remark}

Notice that smooth concordance preserves $(h)$-solvability since if $L_0$ and $L_1$ are smoothly concordant and $L_1$ is $(h)$-solvable, we can construct an $(h)$-solution for $L_0$ in the following way: Do 0-surgery $\times I$ along a concordance from $L_0$ to $L_1$, and glue an $(h)$-solution for $L_1$ to the $M_{L_1}$ boundary component. By a Seifert-van Kampen and Mayer-Vietoris argument, this gives an $(h)$-solution for $L_0$ with the surfaces representing generators for the second homology simply given by the images of the surfaces representing the second homology in the $(h)$-solution for $L_1$.

\begin{remark}
We could have also defined a notion of topological $(h)$-solvability as well. However, when talking about $(h)$-solvability, we will always mean ({\it smooth}) $(h)$-solvability as we have defined it and use smooth versions of concordance/sliceness/etc.
\end{remark}

\begin{definition}
The set of $m$-component $(h)$-solvable links is denoted $\cF^m_{(h)}$. Also define $\cF^m_{(-0.5)}$ to be the set of $m$-component links, all of whose pairwise linking numbers vanish. As mentioned above, it is often convenient to regard such links as ``$(-0.5)$-solvable.''
\end{definition}

This defines a filtration of $\cC^m_\sm$:
	$$\cdots\subseteq\cF^m_{(1.5)}\subseteq\cF^m_{(1)}\subseteq\cF^m_{(0.5)}\subseteq\cF^m_{(0)}\subseteq\cF^m_{(-0.5)}\subseteq\cC^m_\sm$$
This filtration was first defined by Cochran-Orr-Teichner in \cite{cochranKnotConcordanceWhitney2003}. Note that if the surfaces $S_i,T_i$ were spheres, condition (3) of Definition \ref{Solvability} would be trivially satisfied for all $h$. In this case, we could do surgery on one sphere from each pair, replacing a tubular neighborhood $S^2 \times D^2$ with $B^3 \times S^1$ (which have the same boundary), in order to get a manifold which satisfies the conditions of Prop \ref{TopSliceCondition}.

\begin{remark}
When $m=1$, we will drop the superscript and write $\cC_\sm$, $\cC_\top$, and $\cF_{(h)}$.
\end{remark}

All topologically slice knots are $(h)$-solvable for every $h\in\frac{1}{2}\N$ as noted in Section 1 of \cite{cochranKnotConcordanceHigherorder2009}, and it is conjectured that $\bigcap_{h\in\frac{1}{2}\N}\cF_{(h)}$ contains precisely the topologically slice knots. It is known that for all $n\in\N$, the quotient $\cF_{(n)}/\cF_{(n.5)}$ has infinite rank (see \cite{cochranKnotConcordanceHigherorder2009}). However, it is unknown whether $\cF_{(n.5)}/\cF_{(n+1)}$ is even nontrivial for any $n$. In \cite{ottoSolvableFiltrationLink2014}, Otto showed that $\cF_{(n.5)}^{m}/\cF_{(n+1)}^{m}$ contains an infinite cyclic subgroup when $m \geq 3 \cdot 2^{n+1}$.

At points, we will also need the {\it trace} of a link:

\begin{definition}
The {\it trace} $X$ of a link $\Gamma \subset S^3$ is the 4-manifold obtained by attaching 0-framed 2-handles to $B^4$ along $\Gamma$.
\end{definition}

\section{Sufficient conditions for $(n-0.5)$-solvability}
\label{Sec:nSolvability}

In this section we will prove the following theorem:

\ThmnSolv

\begin{remark}
\label{Rmk:WhiskerIndependence}
In order for it to be sensible to talk about curves as fundamental group elements, we have to choose a basepoint and paths from the basepoint to a point on the curve. In the context of basepoints and fundamental groups, such paths are called {\it whiskers}. Note that for any group $G$, the terms $G^{(n)}$ of the derived series are normal in $G$ and so whether a curve belongs to $\pi_1(W)^{(n)}$ does not depend on the choice of basepoint or on the choice of whisker from the basepoint to the curve since this will only change the element of the fundamental group by a conjugation. So, we will often talk about (unbased) curves lying in $\pi_1(W)^{(n)}$.
\end{remark}

\noindent Before proving Theorem \ref{Thm:nSolv}, we will fix the following notation for Sections \ref{Sec:nSolvability} and \ref{Sec:Sliceness}.

Recall that having all pairwise linking numbers 0 is equivalent to the components of $\hat{J}$ bounding disjoint embedded surfaces in $B^4$. Let $X$ be the trace of $\hat{J}$, let $\mu_i$ be fixed meridians of $\hat{J}$ in $M_{\hat{J}} = \partial X$ and $\F$ the natural multi-disk containing them (see Fig.\ \ref{Fig:StringLinkClosure}), let $\Sigma_i$ be disjoint embedded surfaces in $B^4$ bounded by the components of $J$, and $\hat\Sigma_i$ these surfaces capped off by the cores of the 2-handles in $X$.

\begin{figure}
	\includegraphics{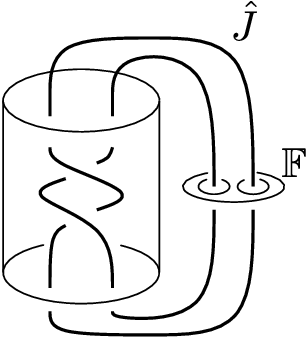}
	\caption{Closure $\hat{J}$ of a string link $J$, with multi-disk $\F$ pictured.}
	\label{Fig:StringLinkClosure}
\end{figure}

Notice that thickening the multi-disks $\E$ and $\F$ with subdisks deleted give handlebodies $\bbH \subset M_L = \partial W$ and $\I \subset M_J = \partial X$, respectively. Note that we can view $\I$ as $S^3 \setminus ((D^2 \times I) \cup \nu\hat{J})$. Let $N$ be the manifold obtained by gluing $X$ to $W$ via identifying $\bbH$ and $\I$ such that the natural meridians and longitudes of each are identified (see Fig.\ \ref{Fig:Solution} for a schematic picture). This gluing ``buries'' the two handlebodies that get identified inside the interior of the new 4-manifold. Moreover, the boundary of the new manifold is $M_P$.

\begin{figure}
	\includegraphics{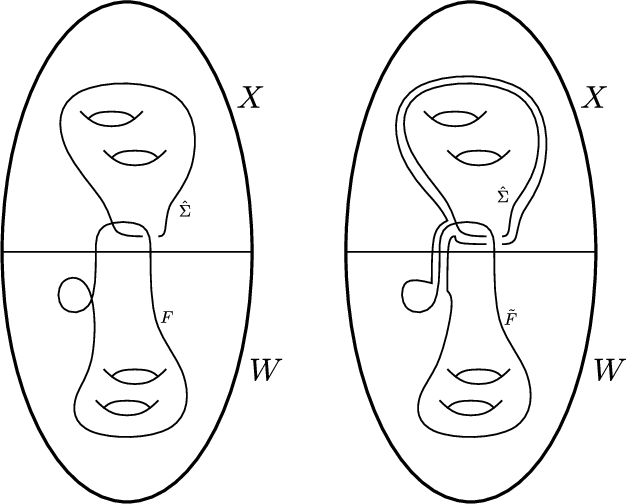}
	\caption{$(n-0.5)$-solution for $P(J)$}
	\label{Fig:Solution}
\end{figure}

\begin{proof}[Proof of Theorem \ref{Thm:nSolv}]
First assume $n \geq 1$. We will construct immersed surfaces in $W$ bounded by the $\eta_i$. The constructions for different $i$ are independent of one another (though the surfaces might intersect), so for each one, choose a basepoint for $W$ on $\eta_i$. Since $\eta_i\in\pi_1(W)^{(n)}$, it can be written as $\prod_k[\alpha_k,\beta_k]$ where $\alpha_k,\beta_k\in\pi_1(W)^{(n-1)}$. Choosing loops representing the $\alpha_k$ and $\beta_k$ gives a map of an annulus coming from a homotopy from $\eta_i$ to the product of the commutators of these representative loops, which we can turn into a map of a surface by doing appropriate identifications on the boundary. After possibly perturbing this map to be in general position, its image is an immersed surface $F_i$ in $W$ with boundary $\eta_i$. Cap off this surface in $X$ using the trivial disk bounded by the unknots $\mu_i$ in $S^3 \subset B^4 \subset X$ (recall that the gluing of $X$ and $W$ identifies $\eta_i$ and $\mu_i$). Notice that $F_i$ and $\hat\Sigma_i$ intersect geometrically once.

The surfaces $\hat\Sigma_i$ are trivially framed since they come from properly embedded surfaces in $B^4$ capped off by core disks of 0-framed 2-handles. And we can arrange that the $F_i$ also be trivially framed by adding local kinks.

Note that the inclusion map $\pi_1(\hat\Sigma)\to\pi_1(X)=1$ is the trivial map, and so in particular lands in every term of the derived series of $\pi_1(X)$, and thus also of $\pi_1(N)$.

We can tube the (self-)intersections among the $F_i$ onto parallel copies of the $\hat\Sigma_i$ to produce embedded surfaces $\tilde{F_i}$ (see Fig.\ \ref{Fig:Solution}).  There is a generating set for $\pi_1(\tilde{F}_i)$ consisting of a symplectic basis for $\pi_1(F_i)$ together with some number of copies of the symplectic bases for $\pi_1(\hat{\Sigma}_i)$, modified by whiskers to the basepoint, which we will now choose to be the intersection point $F_i\cap\hat{\Sigma}_i$ since it is the unique point common to $\hat{\Sigma_i}$, $F_i$, and $\tilde{F_i}$. From the construction of the $F_i$ and Remark \ref{Rmk:WhiskerIndependence}, the basis for $\pi_1(F_i)$ lies in $\pi_1(N)^{(n-1)}$. Since the inclusion map $\pi_1(\hat\Sigma)\to\pi_1(X)=1$ is the trivial map, the basis elements of $\pi_1(\tilde{F}_i)$ coming from the basis for $\pi_1(\hat\Sigma_i)$ are freely nullhomotopic and thus nullhomotopic. Therefore, $\pi_1(\tilde{F})\to\pi_1(N)$ still lands in $\pi_1(N)^{(n-1)}$.

Thus, $N$ is an $(n-0.5)$-solution for $P$, giving Theorem \ref{Thm:Slice} when $n \geq 1$, modulo the following lemma:

\begin{lemma}
\label{Lem:nSoln}
\begin{enumerate}
	\item $\pi_1(N)$ is normally generated by the meridians of $P=P(J)$
	\item $H_1(N)\cong\Z^m$ and is generated by the meridians of $P$.
	\item $H_2(N)\cong\Z^{2r}$ and is generated by the $\tilde{F_i}$ and $\hat\Sigma_i$.
\end{enumerate}
\end{lemma}

\begin{proof}
	(1) $\pi_1(W)$ is normally generated by the meridians of $L$. In $N$, a meridian of $L$ is freely homotopic to  the corresponding meridian of $P$. Moreover, $\pi_1(X)=1$, and so by Seifert-van Kampen, $\pi_1(N)$ is normally generated by the meridians of $P$.
	
	(2 and 3) We have the following Mayer-Vietoris sequence:
	$$\begin{tikzcd}
		0 \arrow[r]
			& 0 \oplus H_2 (X) \arrow[r] \arrow[d, phantom, ""{coordinate, name=Z}]
			& H_2(N)  \arrow[dll,
											rounded corners,
											to path={ -- ([xshift=2ex]\tikztostart.east)
											|- (Z) [near end]\tikztonodes
											-| ([xshift=-2ex]\tikztotarget.west)
											-- (\tikztotarget)}] \\
		H_1(\nu(\E)) \arrow[r]
			& H_1(W) \oplus 0 \arrow[r]
			& H_1(N) \arrow[r]
			& 0
	\end{tikzcd}.$$
The map $H_1(\nu(\E)) \to H_1(W)$ is the zero map, since the $\eta_i$ lie in $\pi_1(W)^{(n)}$, and in particular in $\pi_1(W)^{(1)}$, and so are nullhomologous. This implies that the map $H_1(W) \to H_1(N)$ is an isomorphism, giving (2) after noting that the generators of $H_1(W)$ are the meridians of $L$, which are freely homotopic through $N$ to the meridians of $P$ in $M_P = \partial N$.

Moreover, $H_2(N) \cong H_2(X) \oplus H_1(\nu(\E))$. The generators of $H_2(N)$ from $H_2(X)$ are given simply by the $\hat\Sigma_i$, and the generators of $H_2(N)$ from $H_1(\nu(\E))$ can be represented by any immersed surface in $W$ bounded by $\eta_i$ glued to any immersed surface in $X$ bounded by the corresponding $\mu_i$ (giving a single immersed closed surface since $\eta_i$ and $\mu_i$ are identified). The surface $F_i$ is such a surface, and tubing away the self-intersections simply amounts to a change of basis for $H_2(N)$, giving (3).
\end{proof}

Now, suppose that $n=0$, or in other words that we drop any restriction  on the $\eta_i$. We then wish to show that $P(J)$ has pairwise vanishing linking numbers. Recall that given two components $K_0$ and $K_1$ of a link, the linking number $\lk(K_0,K_1)$ is defined to be the number of times $K_0$ crosses {\it under} $K_1$, counted with sign after choosing a diagram for $K_0 \sqcup K_1$. Choose diagrams for $J$ and for $L\sqcup\eta_1\sqcup\cdots\sqcup\eta_r$ such that in the diagram for $L\sqcup\eta_1\sqcup\cdots\sqcup\eta_r$, the link $\eta_1\sqcup\cdots\sqcup\eta_r$ appears as the standard diagram of the unlink, and the only crossings of this standard unlink with $L$ appear as strands of $L$ going through the $\eta_i$ as in Figure \ref{Fig:Multiinfection}. We can then form a diagram for $P$ by plugging in a diagram consisting of parallel copies of the components of $J$ where the $\eta_i$ appear in the diagram for $L\sqcup\eta_1\sqcup\cdots\sqcup\eta_r$. We may have to add twists in order to make sure that the parallel copies of the components of $J$ have linking number 0 with one another. One can then check that because $J$ has pairwise vanishing linking numbers, the linking numbers of $P$ are the same as the linking numbers of $L$. And since $L$ is a slice link, its linking numbers vanish pairwise. So $P$ is ``$(-0.5)$-solvable.''
\end{proof}

\begin{remark}
Notice that when $\eta_1,...,\eta_r$ are freely nullhomotopic in $W$, the manifold $W$ is an $(n)$-solution for $P$ for every $n\in\N$ (with possibly different surfaces witnessing the $(n)$-solvability for each $n$). In particular, Theorem \ref{Thm:nSolv} shows that the Whitehead double of the Borromean rings is $(n)$-solvable for every $n$ (or more generally the Whitehead double of any link with vanishing linking numbers). See Figure \ref{Fig:WhBR}. This fact is already well-known to the experts, in fact the manifold $W$ constructed in Section 5 of \cite{chaFamilyFreelySlice2020} is an $(n)$-solution for every $n$ for any good boundary link (although in this paper they use this manifold to show that every good boundary link with a Seifert surface admitting a homotopically trivial$^+$ good basis is freely topologically slice). The Whitehead double of the Borromean rings is an important test case: if it is freely topologically slice, then it is likely that the Disk Embedding Theorem holds without any fundamental group condition. For more information, see Sections 12.1-3 in \cite{freedmanTopology4manifolds1990} and Sections 23.1 and 23.2 of \cite{behrensDiskEmbeddingTheorem2021}. Theorems \ref{Thm:COTThm} and \ref{Thm:Slice} do not apply to this link since in the satellite construction, we have $2c=c+r=6$, but the Milnor's invariants of the Borromean rings only vanish up to length 2. In principle, improvements to theorems of this case might show that the Whitehead double of the Borromean rings is topologically slice.
\end{remark}

\begin{figure}
\includegraphics[scale=0.5] {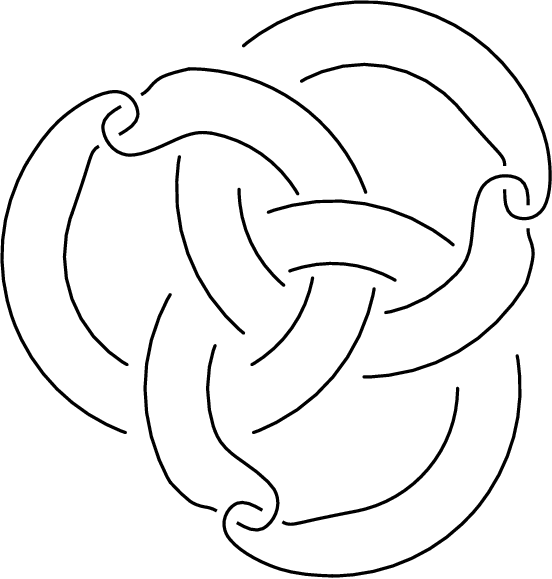}
\caption{The Whitehead double of the Borromean rings}
\label{Fig:WhBR}
\end{figure}

\section{Proof of the main theorem}
\label{Sec:Sliceness}

In this section, we will prove the main theorem of this paper:

\ThmSlice

Before proving Theorem \ref{Thm:Slice}, we will review the definition of the algebraic (self-)intersection numbers, following Section 11.3 in \cite{behrensDiskEmbeddingTheorem2021}, along with some definitions and a result from Chapters 5 and 6 \cite{freedmanTopology4manifolds1990}, and a result from \cite{cochranDerivativesLinksMilnor1990} (see also \cite{freedman4ManifoldTopologyII1995}).

Let $M$ be a connected based smooth oriented 4-manifold, and let $f$ and $g$ be generically, properly, and smoothly immersed based oriented spheres. We will be talking about curves on $f \cup g$ as elements of $\pi_1(M)$, so as in Remark \ref{Rmk:WhiskerIndependence}, we must choose whiskers $v_f$ and $v_g$ from the basepoint of $M$ to $f$ and $g$, respectively.

\begin{definition}
The {\it algebraic intersection number} $\lambda(f,g)$ is the element of $\Z[\pi_1(M)]$ defined by
	$$\lambda(f,g) := \sum_{p \in f \pitchfork g} \epsilon(p) \alpha(p),$$
where
\begin{itemize}
	\item $\alpha(p) = v_f \gamma_f^p (\gamma_g^p)^{-1} v_g^{-1} \in \pi_1(M)$,
	\item $\gamma_f^p$ (resp. $\gamma_g^p$) is a simple path in $f$ (resp. $g$) from the basepoint of $f$ (resp. $g$) to $p$ which does not hit any other points in $f \pitchfork g$, and
	\item $\epsilon(p) \in \{\pm 1\}$ is the sign of the intersection at $p$.
\end{itemize}
See Figure \ref{Fig:IntersectionNumberDef} for a schematic picture. Also define $\lambda(f,f):=\lambda(f,f^+)$ where $f^+$ is a transverse pushoff of $f$. Notice that distinct choices of $\gamma_f^p$ and $\gamma_g^p$ are homotopic since spheres are simply connected, and so this is well-defined.
\end{definition}

\begin{figure}
	\includegraphics{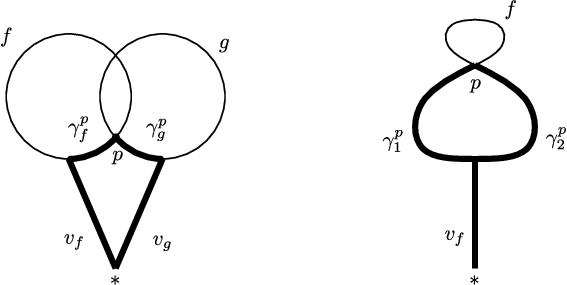}
	\caption{Schematic pictures of the algebraic intersection number (left) and algebraic self-intersection number (right)}
	\label{Fig:IntersectionNumberDef}
\end{figure}

See Proposition 11.4 in \cite{behrensDiskEmbeddingTheorem2021} for a proof that $\lambda(f,g)$ does not change under regular homotopies of $f$ and $g$, changes by a sign when the orientation of $f$ or $g$ is reversed, and changes in a controlled way under different choices of whiskers.

There is a natural involution on $\Z[\pi_1(M)]$ induced by inversion on $\pi_1(M)$, which we will denote by $\bar\cdot$.

\begin{definition}
The {\it algebraic self-intersection number} $\mu(f)$ is the element of $\Z[\pi_1(M)]/(a\sim\bar{a})$ defined by
	$$\mu(f) := \sum_{p \in f \pitchfork f} \epsilon(p)\alpha(p),$$
where
\begin{itemize}
	\item $\alpha(p) = v_f \gamma_0^p (\gamma_1^p)^{-1} v_f^{-1} \in \pi_1(M)$,
	\item $\gamma_0^p$ and $\gamma_1^p$ are simple paths in $f$ from the basepoint to $p$ along two different sheets which do not hit any other points in $f \pitchfork f$, and
	\item $\epsilon(p) \in \{\pm 1\}$ is the sign of the self-intersection at $p$.
\end{itemize}
See Figure \ref{Fig:IntersectionNumberDef} for a schematic picture. See Proposition 11.7 in \cite{behrensDiskEmbeddingTheorem2021} for a proof that this is well-defined.
\end{definition}

When $f$ is trivially framed, Proposition 11.8 in \cite{behrensDiskEmbeddingTheorem2021} states that $\lambda(f,f)=\mu(f)+\overline{\mu(f)}$ (note that a lift of $\mu(f)+\overline{\mu(f)}$ to $\Z[\pi_1(M)]$ is independent of the choice of lift of $\mu(f)$, and so this gives a well-defined element of $\Z[\pi_1(M)]$ even though {\it a priori} $\lambda(f,f)$ and $\mu(f)+\overline{\mu(f)}$ live in different groups). For our purposes in this paper, we will always be able to add local kinks when needed in order to ensure that $f$ is trivially framed, and so by Corollary 11.9 in \cite{behrensDiskEmbeddingTheorem2021}, we will have $\lambda(f,f)=0\Leftrightarrow\mu(f)=0$.

The following definitions can be found in Section 5.3 of \cite{freedmanTopology4manifolds1990}.

\begin{definition}
An $h$-cobordism rel boundary between two manifolds $M^n$ and $N^n$ with homeomorphic (possibly empty) boundary $P$ is a manifold $W^{n+1}$ with boundary $M \cup (P \times I) \cup N$ where $M$ and $N$ are glued to $P \times I$ along their boundaries, such that $W$ deformation retracts to $M$ and to $N$. Given homeomorphic properly immersed submanifolds $A \looparrowright M$, $B \looparrowright N$, an $h$-cobordism rel boundary of the pairs $(M,A)$, $(N,B)$ is an $h$-cobordism rel boundary $W$ of $M$ and $N$ together with a proper immersion $A \times I \looparrowright W$ which is a product on $P \times I$ and such that the restriction to $A\times\{0\}$ gives $A \looparrowright M$ and the restriction to $A \times \{1\}$ gives $B \looparrowright N$. See Figure \ref{fig:PairHCobordism} for a schematic picture.
\end{definition}

\begin{figure}
	\includegraphics{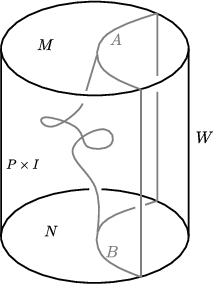}
	\caption{A schematic picture of an $h$-cobordism rel boundary of pairs. Note that $A$ and $B$ may also be immersed, but must be homeomorphic.}
	\label{fig:PairHCobordism}
\end{figure}

\begin{definition}
An $s$-cobordism (rel boundary of pairs) is an $h$-cobordism (rel boundary of pairs) such that the Whitehead torsion $\tau(W,M)$ vanishes. For a definition and discussion of the Whitehead torsion, see Section IV of \cite{cohenCourseSimplehomotopyTheory1973} (for a definition and discussion in the smooth setting, see \cite{milnorWhiteheadTorsion1966}).
\end{definition}

Although Theorem \ref{Thm:sCobordism} will guarantee us an $s$-cobordism, we will only use the fact that this $s$-cobordism is an $h$-cobordism, and in fact only that it gives a homotopy equivalence where $A$ and $B$ represent the same homology classes.

The algebraic (self-)intersection numbers show up in the following theorem which we will use in the proof of Theorem \ref{Thm:Slice}:

\begin{theorem}[Theorem 6.1 in \cite{freedmanTopology4manifolds1990}]
\label{Thm:sCobordism}
A $\pi_1$-null immersion of a union of transverse pairs of immersed spheres with algebraically trivial (self-)intersection numbers is $s$-cobordant rel boundary to a topological embedding.
\end{theorem}

The Milnor's invariant condition will show up via the following theorem:

\begin{theorem}[Corollary 9.4 in \cite{cochranDerivativesLinksMilnor1990}, see also Lemma 2.7 in \cite{freedman4ManifoldTopologyII1995}]
\label{Thm:MilnorImmersedDisks}
For a link $\Gamma$ and $\ell \geq 2$, the following are equivalent:
\begin{itemize}
\item all Milnor's $\bar\mu$-invariants up to length $\ell$ vanish for $\Gamma$
\item if $\tilde{\Gamma}$ is a $\ell$-component link consisting of parallel copies of components of $J$, then $\tilde{\Gamma}$ bounds disjoint immersed disks in $B^4$.
\end{itemize}
\end{theorem}

We will now prove the main theorem of this paper.

\begin{proof}[Proof of Theorem \ref{Thm:Slice}]
Let $N$ be the same manifold constructed in section \ref{Sec:nSolvability} by gluing the trace $X$ of $\hat{J}$ to $W$ along handlebody thickenings of $\E \subset \partial W$ and the standard multi-disk $\F \subset M_L = \partial W$. Recall that $\mu_i$ are fixed meridians of $\hat{J}$ in $M_{\hat{J}}=\partial X$, and the $\eta_i$ are boundaries of the $E_i\subset\E_i$, and are identified with the $\mu_i$ under the gluing. We will find immersed spheres rather than embedded surfaces representing a basis of $H_2(N)$, which we can then surger out using \cite{freedmanTopology4manifolds1990}. These immersed spheres will come in geometrically transverse pairs, and in order to ensure this, we must construct them in some sense simultaneously.

Recall that by assumption the curves $\eta_i$ bound immersed disks $\delta_i$ in $W$, which together have a total of $c$ (self-)intersections. Since these disks are nullhomotopies in $B^4 \supset W$ of curves in $S^3$, their framings must be even.

The disks $\delta_i$ may be capped off inside $X$ using the trivial disks bounded by the $\mu_i$ to obtain immersed spheres $A=\alpha_1\cup\cdots\cup\alpha_r$. We will eventually want to modify these spheres to be $\pi_1$-null. To do this we will begin by tubing the (self-)intersections in $A$ to the gluing region between $W$ and $X$, as if we were going to tube on a sphere, but right now we have no sphere to tube on. Notice that this gives a $c$-component link consisting of parallel copies of components of $\hat{J}$. We consider this link together with the attaching circles of the 2-handles in $X$, which gives a $($\hl{$c+r$}$)$-component link consisting of parallel copies of components of $\hat{J}$. This link bounds disjoint immersed disks in $B^4 \subset X$ since the Milnor $\bar\mu$-invariants of $\hat{J}$ vanish up to length \hl{$c+r$} by Theorem \ref{Thm:MilnorImmersedDisks} (Lemma 2.7 in \cite{freedman4ManifoldTopologyII1995}). Since these disks are nullhomotopies in $B^4$ of curves in $S^3$, their framings are even. Cap off the components of this link with these disks to obtain immersed spheres $\tilde{A}=\tilde\alpha_1\cup\cdots\cup\tilde\alpha_r$ and (geometrically) transverse spheres $B=\beta_1\cup\cdots\cup\beta_r$ coming from disks bounded by one copy of each component of $\hat{J}$ plus the cores of the 2-handles. Notice that of the disjoint immersed disks guaranteed by the Milnor's invariant condition, $c$ of them are used in forming the $\tilde\alpha_i$ and $r$ of them are used in forming the $\beta_i$. The spheres $A$ and $B$ form a symplectic basis of $H_2(N)$, as do $\tilde{A}$ and $B$, by the same reasoning as in Lemma \ref{Lem:nSoln}. Moreover, $\tilde{A}$ and $B$ must have even framings, and we can add local kinks inside $B^4 \subset X$ to make the framings trivial. See Fig.\ \ref{Fig:ImmersedSpheres} for a schematic pictures.

\begin{figure}
\includegraphics{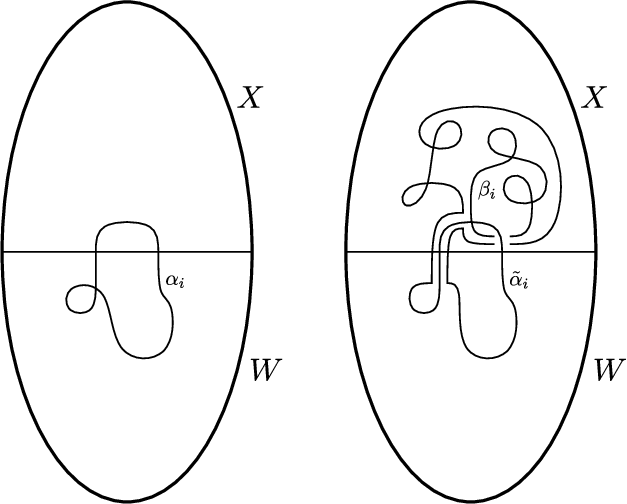}
\caption{Schematic of $\tilde{A}=\tilde\alpha_1\cup\cdots\cup\tilde\alpha_r$ and $B=\beta_1\cup\cdots\cup\beta_r$ inside $N$}
\label{Fig:ImmersedSpheres}
\end{figure}

First we will check that the collection of spheres $\tilde{A} \cup B$ is $\pi_1$-null. The $\tilde\alpha_i$ and $\beta_i$ are already geometrically dual, and so we only need to check that loops in each of the $\tilde\alpha_i$ and $\beta_i$ are freely nullhomotopic in $N$. The $\beta_i$ lie entirely in $X$, which is simply connected, so any loops in the $\beta_i$ are trivial. A schematic picture of the $\tilde\alpha_i$ is show in Fig.\ \ref{Fig:tildealphaiSchematic}. The portion of $\tilde\alpha_i$ which lies outside $B^4$ is a several-times punctured sphere, which deformation retracts to a wedge of the boundary circles, and so $\tilde\alpha_i$ deformation retracts to a wedge of the disjoint immersed disks in $B^4$. Therefore any loop  in $\tilde\alpha_i$ is a product of loops entirely contained in these disjoint immersed disks, and so is freely nullhomotopic in $B^4$, and so the whole loop is freely nullhomotopic in $N$. Thus, the $\tilde\alpha_i$ are $\pi_1$-null.

The same argument gives that the algebraic self-intersection numbers of the $\tilde\alpha_i$ and $\beta_i$ are all 0. (If the collections $\tilde{A}$ and $B$ were not already geometrically dual, we would have more work to do.)

The algebraic intersection numbers $\lambda(\tilde\alpha_i,\tilde\alpha_j)$, $\lambda(\beta_i,\beta_j)$, and $\lambda(\tilde\alpha_i,\beta_j)$ for $i \neq j$ are 0 since these pairs of spheres are disjoint.

\begin{figure}
\includegraphics[scale=0.5]{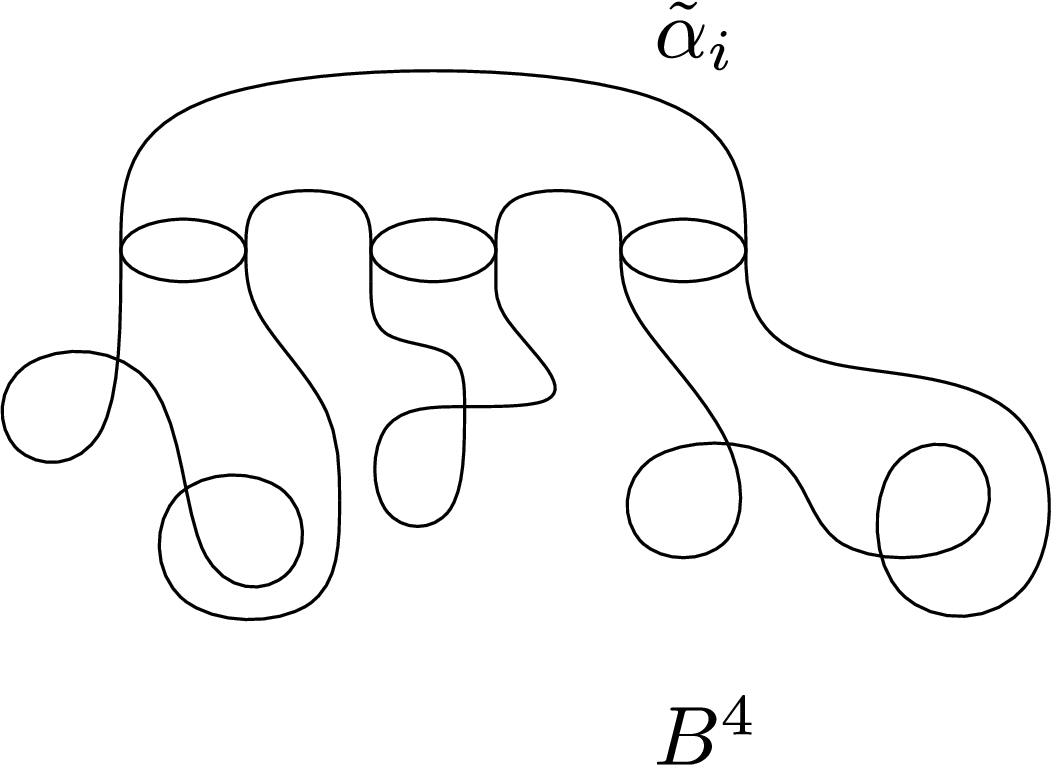}
\caption{Computation of the self-intersection numbers of the $\tilde\alpha_i$}
\label{Fig:tildealphaiSchematic}
\end{figure}

Therefore, by Theorem \ref{Thm:sCobordism} (Thm.\ 6.1 in \cite{freedmanTopology4manifolds1990}) there is an $s$-cobordism rel boundary of the pair $(N, \tilde{A} \cup B)$ to a pair $(N',\text{embedded geometrically dual spheres})$ where the embedded spheres represent the same homology classes and still come in pairs which intersect geometrically once transversely. We can then do surgery on one sphere from each pair to obtain a manifold which satisfies the conditions of Proposition \ref{TopSliceCondition}. Therefore the multi-infection $P$ is topologically slice.
\end{proof}

\begin{remark}
\label{Rmk:Degenerate}
If any of the $\delta_i$ are embedded and disjoint from the remaining $\delta_i$, then the corresponding sphere $\alpha_i$ is embedded (see Figure \ref{Fig:ImmersedSpheres}). In this case, $\alpha_i=\tilde\alpha_i$, and $\beta_i$ can be made embedded by tubing copies of $\alpha_i$ onto one sheet of each self-intersection. We thus obtain an embedded geometrically dual pair of spheres without using Theorem \ref{Thm:sCobordism}. The rest of the proof then proceeds the same way. We can use this technique to first reduce to the case where $r \leq c$, where Theorem \ref{Thm:Slice} is stronger than Theorem \ref{Thm:COTThm} (but still not strictly stronger). The original proof of Theorem \ref{Thm:COTThm} can similarly be modified to avoid using such powerful tools on any of the $\delta_i$ which are embedded and disjoint from the other $\delta_i$.
\end{remark}

\begin{remark}
\label{Rmk:MilnorRefinement}
A more refined statement of Theorem \ref{Thm:MilnorImmersedDisks} (Corollary 9.4 in \cite{cochranDerivativesLinksMilnor1990}) is the following theorem: (which appears as Proposition 9.3 in \cite{cochranDerivativesLinksMilnor1990}):
\begin{theorem}[Proposition 9.3 in \cite{cochranDerivativesLinksMilnor1990}]
For an $r$-component link $\Gamma$ and $\ell \geq 2$, and $\ell_1,...,\ell_r \in \Z^+$ such that $\ell_1+\cdots+\ell_r=\ell$, the following are equivalent:
\begin{itemize}
\item Milnor's invariants $\bar\mu_\Gamma(I)=0$ for all sequences $I$ containing no more than $\ell_i$ occurrences of the index $i$.
\item The link $\tilde{\Gamma}$, which consists of $\ell_i$ parallel copies of the $i$th component of $\Gamma$, bounds disjoint immersed disks in $B^4$
\end{itemize}
\end{theorem}
We could have used this theorem rather than Theorem \ref{Thm:MilnorImmersedDisks} in the proof of Theorem \ref{Thm:Slice} to obtain a refined Milnor's invariant condition. In particular, define $\ell_i$ as follows:
\begin{enumerate}[label=(\roman*)]
\item start with $\ell_1=\cdots=\ell_r=1$, then
\item for each self-intersection of the disk bounded by $\eta_i$, add 1 to $\ell_i$, and
\item for each intersection between the disks bounded by $\eta_i$ and $\eta_j$ for $i \neq j$, add 1 to either $\ell_i$ or $\ell_j$.
\end{enumerate}
Then if $\bar\mu_J(I)=0$ for all $I$ containing no more than $\ell_i$ occurrences of $i$, then the multi-infection of $L$ by $J$ is still topologically slice. Here, for each intersection between disks, we made a choice, and as long as the Milnor's invariant condition holds for at least one of our possible choices, the conclusion still holds.
\end{remark}

\section{Example of links which are topologically slice by Theorem \ref{Thm:Slice}, but not by Theorem \ref{Thm:COTThm}}
\label{Sec:Ex}

In this section, we will describe an example of a link which is topologically slice by Theorem \ref{Thm:Slice}, but not by Theorem \ref{Thm:COTThm}. We will need the following lemma:

\begin{lemma}
\label{Lemma:Surfaces}
Let $H$ be a link with pairwise linking numbers 0 which bounds a collection of immersed disks with $c$ (self-)intersections. Then the components of $H$ bound disjoint embedded surfaces, the sum of whose genera is $\leq c$. In other words, if $c(H)$ is the minimal number of (self-)intersections in a collection of immersed disks bounded by $H$, then we have the following inequality:
	$$g_4(H) \leq c(H),$$
where $g_4$ is the slice genus.
\end{lemma}

\begin{remark}
The previous lemma holds in both the smooth and topological settings, and the proof is the same regardless.
\end{remark}

\begin{proof}
	Since $H$ has pairwise linking number 0, intersections between disks bounded by different components of $H$ must come in pairs with opposite signs, which can then be tubed together. This process increases the genus of the resulting surface by 1 each time. Any self-intersections can be resolved by deleting a neighborhood of the intersection point and gluing in a twisted annulus whose boundary is the resulting Hopf link, which also increases the genus by 1 each time. The sum of the genera of the resulting surfaces is then $\leq c$ (it is also $\geq\frac{c}{2}$, though we won't use this). See Fig.\ \ref{Fig:Surfaces} for a schematic picture.
\end{proof}

\begin{figure}
	\includegraphics{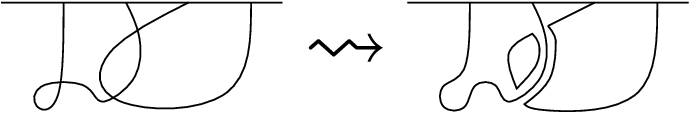}
	\caption{Resolving (self-)intersections using tubes and twisted annuli}
	\label{Fig:Surfaces}
\end{figure}

\begin{figure}
	\includegraphics[scale=0.5]{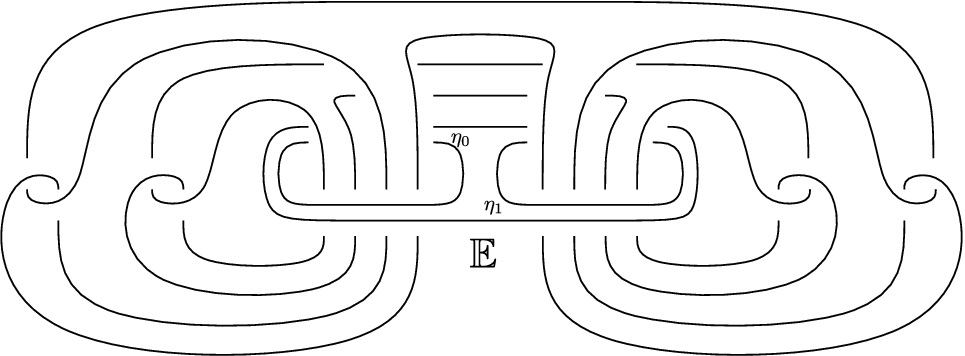}
	\caption{A proper multidisk $\E$ in the complement of a $2$-component unlink $L$, such that the curves $\eta_1$ and $\eta_2$ are nullhomotopic in the complement of some collection of slice disks for $L$, but do not bound immersed disks with fewer than 3 total (self-)intersections}
	\label{Fig:Example}
\end{figure}

\begin{figure}
	\includegraphics[scale=0.5]{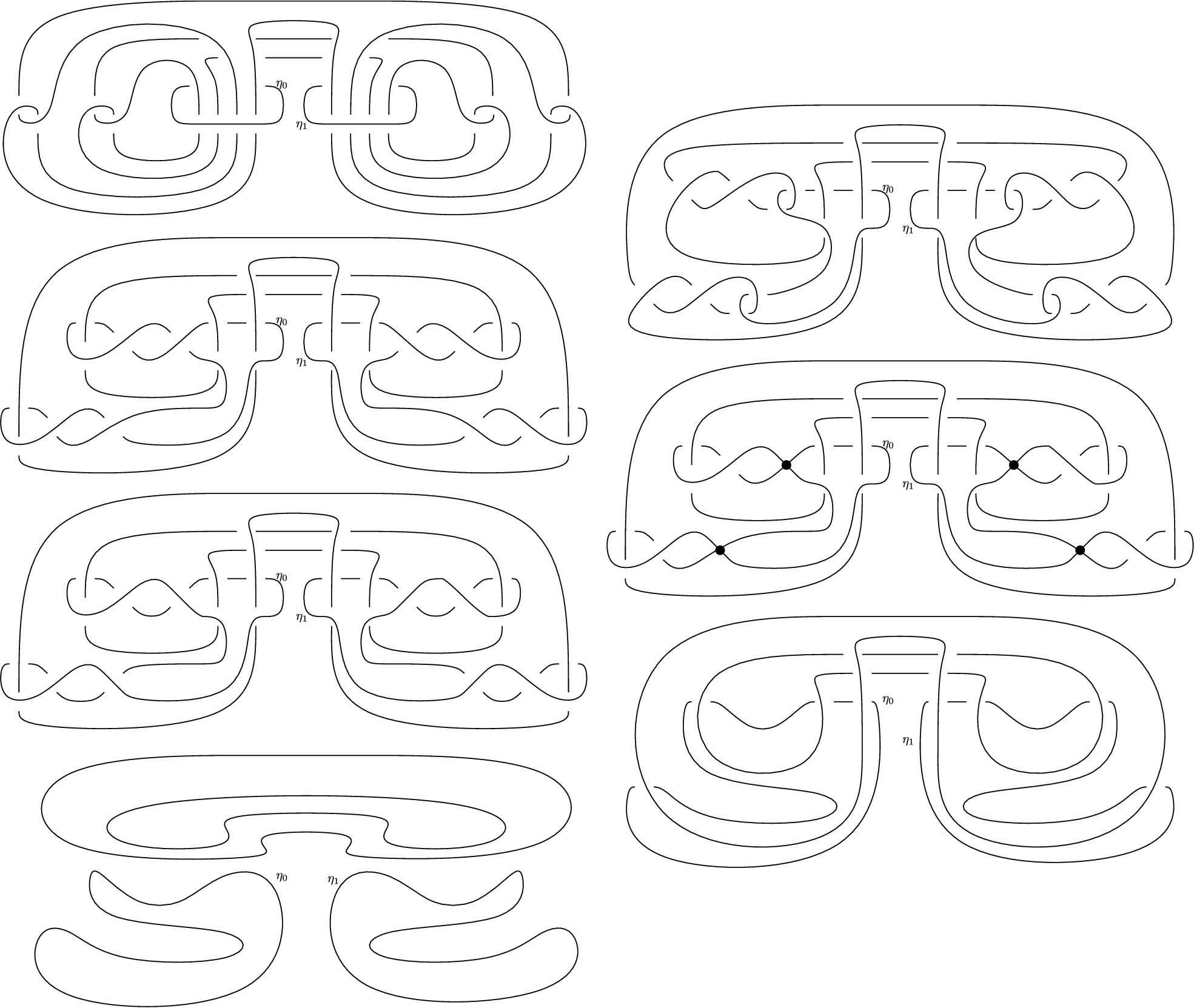}
	\caption{A sequence of diagrams showing that the $\eta_i$ are nullhomotopic in the complement of slice disks for $L$ (in fact, just in the complement of $L$ in $S^3$). Self-intersections during the course of this nullhomotopy are labeled by dots. The sequence should be read top to bottom, zigzagging left to right down the page.}
	\label{Fig:ExampleNullhomotopy}
\end{figure}

\begin{figure}
	\includegraphics[scale=0.5]{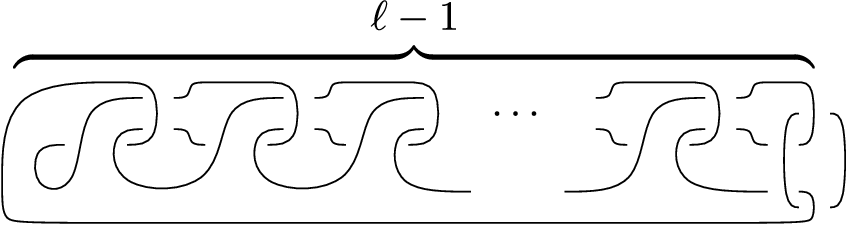}
	\caption{A 2-component link whose Milnor's $\bar\mu$-invariants vanish up to length $\ell$, but not up to length $\ell+1$.}
	\label{Fig:Brunnian}
\end{figure}

\begin{ex}
We will now demonstrate an example of a link which is topologically slice by Theorem \ref{Thm:Slice} but not by Theorem \ref{Thm:COTThm} by doing multi-infection on the link $L$ in Fig.\ \ref{Fig:Example} along the mutli-disk pictured there. Nullhomotopies of $\eta_0$ and $\eta_1$ in the complement of slice disks for $L$ (in fact, just in the complement of $L$ in $S^3$), are shown in Figure \ref{Fig:ExampleNullhomotopy}. These nullhomotopies have 4 (self-)intersections.

Using {\tt SnapPy} (see \cite{cullerSnapPyComputerProgram23}), one can compute the signature of the link $L\cup\eta_1\cup\eta_2$ to be 3, and the Alexander polynomial to be $-4t^{-5}+36t^{-4}-144t^{-3}+336t^{-2}-504t^{-1}+504-336t+144t^2-35t^3+4t^4$. Then, by Corollary 1.5 in \cite{powellFourgenusLinkLevine2017}, the topological slice genus of this link is at least 3. And so, by the proof of Lemma \ref{Lemma:Surfaces}, in the complement of any collection of topological slice disks for $L$, immersed disks bounded by the $\eta_i$ must collectively have at least 3 (self-)intersections. So, in summary, any immersed disks bounded by the $\eta_i$ in the complement of a collection of topological slice disks for $L$ must have either 3 or 4 (self-)intersections

Thus by the discussion in Section \ref{Sec:Intro}, if immersed disks bounded by the $\eta_i$ have 3 (resp.\ 4) (self-)intersections, the multi-infection by a string link whose closure has Milnor's $\bar\mu$-invariants vanishing up to length 5 (resp.\ 6) but not up to length 6 (resp.\ 7) is topologically slice by Theorem \ref{Thm:Slice} but not by Theorem \ref{Thm:COTThm}. In Sections 7.2-4 of \cite{cochranDerivativesLinksMilnor1990}, for any $r>1$ and $\ell \geq 0$, Cochran constructs $r$-component links whose Milnor $\bar\mu$-invariants vanish up to length $\ell$ but not up to length $\ell+1$. One such class of links when $r=2$ is shown in Fig.\ \ref{Fig:Brunnian}. (As an aside, notice that this link is not topologically slice, or even (1)-solvable, for example by Corollary 3.5 in \cite{ottoSolvableFiltrationLink2014}.) Multi-infection by any string link whose closure is this link will be topologically slice by Theorem \ref{Thm:Slice} but not by Theorem \ref{Thm:COTThm}.
\end{ex}

\begin{remark}
This is a fairly flexible construction. Not only can one choose any string link whose closure is the the link in Figure \ref{Fig:Brunnian}, there are many different links that can come out of Cochran's construction which will have Milnor's $\bar\mu$-invariants vanishing up to length $\ell$ but not up to length $\ell+1$. Even more, notice that we only used facts about $\eta_0$ and $\eta_1$, and there are many multi-disks $\E$ which have these $\eta_i$, even without varying the link $L\sqcup\eta_0\sqcup\eta_1$.
\end{remark}

\bibliography{SatelliteSolvabilityNote}
\bibliographystyle{alpha}

\end{document}